\documentclass[final]{siamltex}

\usepackage{amsmath,amsfonts,amssymb}

\newcommand{\real}{{\mathbb R}}
\newcommand{\posint}{{\mathbb Z}_+}
\newcommand{\integ}{{\mathbb Z}}
\newcommand{\nat}{{\mathbb N}}
\newcommand{\rat}{{\mathbb Q}}

\newcommand{\sF}{{\mathcal F}}
\newcommand{\sL}{{\mathcal L}}

\newcommand{\sS}{{\mathcal S}}
\newcommand{\sA}{{\mathcal A}}
\newcommand{\scP}{{\mathbb P}}

\title{Moment growth bounds on continuous time Markov processes on non-negative integer lattices
} 
\author{Muruhan Rathinam\thanks{ Mathematics and Statistics, University of
Maryland Baltimore County, 1000 Hilltop Circle, Baltimore, MD 21250 ({\tt muruhan@umbc.edu}, Ph 410-455-2423, Fax 410-455-1066. The research of this author was
supported by grant NSF DMS-0610013.} }

\begin{document}

\maketitle
\date{}

\begin{abstract}
We consider time homogeneous Markov processes in continuous time with state space $\posint^N$ and provide two sufficient conditions and one necessary condition for the existence of moments $E(\|X(t)\|^r)$ of all orders $r \in \nat$ for all $t \geq 0$. 
The sufficient conditions also guarantee an exponential in time growth bound for the moments.   
The class of processes studied have finitely many state independent 
jumpsize vectors $\nu_1,\dots,\nu_M$. This class of processes arise 
naturally in many applications such as stochastic models of chemical kinetics, 
population dynamics and epidemiology for example. We also provide a necessary and sufficient condition for stochiometric boundedness of species in terms of $\nu_j$.  
\end{abstract}

\begin{keywords} 
Continuous time jump Markov process, moment growth bounds, stochastic chemical kinetics, stochastic population dynamics. 
\end{keywords}

\begin{AMS}
60J27.
\end{AMS}

\thispagestyle{plain}

\section{Introduction}
Time homogeneous Markov processes in continuous time with the non-negative integer lattice 
as state space arise in stochastic models of chemical kinetics, 
predator-prey systems, and epidemiology etc. While the primary 
focus of this paper shall be the Markov processes describing stochastic 
chemical kinetics, the results derived will be of use in other applications 
where the processes have similar structure.  More specifically any time homogeneous Markov 
process model that evolves in continuous time on the 
state space $\posint^N$ and has finitely many 
types of jump events with fixed (state and time independent) jump sizes $\nu_1,\dots,\nu_M$ will be the subject of study in this paper.  

A stochastic chemical system with $N \in \nat$ species and $M \in \nat$ reaction channels is described by a Markov process $X(t)$ in continuous time $t \geq 0$ with state space $\posint^N$. The $i$th component $X_i(t)$ describes the (random) number
of species at time $t$.  The probability law of the process is uniquely 
characterized by the {\em stoichiometric matrix} $\nu$ which is 
$N \times M$ with integer entries and the {\em propensity function} 
$a:\posint^N \to \real_+^M$. The functions $a_j:\posint^N \to \real_+$ are also known as {\em intensity functions} or {\em rate functions}. We shall use the 
term propensity which is used in chemical kinetics. The function $a_j(x)$ 
describes the ``probabilistic rate'' at which reaction $j$ occurs while in 
state $x$.  More precisely, given $X(t)=x$, the probability that reaction $j$ 
occurs during $(t,t+h]$ is given by $a_j(x) h + o(h)$ as $h \to 0+$. 
Column vectors of $\nu$ are denoted by $\nu_j$ for $j=1,\dots,M$, and $\nu_j$ 
describes the change of state due to one occurrence of reaction $j$.  
See \cite{Van-book,Gillespie1977} for general introduction to stochastic models in chemical kinetics.

As an example consider the system with $N=2$ species $S_1, S_2$ and $M=2$ reactions given by  
\[
S_1 \to S_2, \quad S_2 \to S_1.
\]
Here the first reaction is one where one $S_1$ is converted into one $S_2$ and 
the second reaction is precisely the reversal of the first. The stoichiometric 
vectors are given by $\nu_1 = (-1,1)^T$ and $\nu_2 = (1,-1)^T$. In the 
standard model of chemical kinetics the propensity functions for this example 
are given by $a_1(x) = c_1 x_1$ and $a_2(x) = c_2 x_2$ and in general 
the propensity functions are derived from combinatorial considerations and 
hence are polynomials\cite{Gillespie1977}. In this paper however we allow more general form for the propensities as we do not want to limit ourselves to models 
arising in chemical kinetics.   
 
The time evolution of the probabilities $p(t;x) = \text{Prob}(X(t)=x)$ is 
governed by the Kolomogorov's forward equations 
\begin{equation}
\frac{d}{dt} p(t;x) = \sum_{j=1}^M [a_j(x-\nu_j) p(t;x-\nu_j) - a_j(x) p(t;x)],
\label{eq_Kol}
\end{equation}
where the functions $a_j$ are understood to be zero if $x-\nu_j \notin \posint^N$, and this is typically an infinite system of ODEs indexed by $x \in \posint^N$. 
While the initial condition in general may be an arbitrary initial distribution $p(0;x)$ on $\posint^N$, it is adequate to study the case of deterministic initial conditions, i.e.\ $p(0;x) = \delta_{x_0}(x)$ in order to make conclusions about the general case. 

In many practical examples, the system is bound to stay in a finite 
subset of $\posint^N$ which is determined by the initial state $x_0$. 
In the above example $S_1 \to S_2, S_2 \to S_1$, it is clear that the total 
number of species $X_1(t)+X_2(t)$ is conserved for all time $t \geq 0$. As a 
result the system shall remain in a finite subset of $\posint^N$. While such 
conservation laws and the consequent boundedness of the system 
are easy to spot for small systems, it may be difficult to decide for 
a large system. In this paper we develop a systematic theory of boundedness 
of species and provide necessary and sufficient conditions based on results from the study of convex cones in finite dimensions. These conditions are expressed in terms of the solution of linear inequalities which can be formulated as a linear programming problem for which several solution techniques exist \cite{Vanderbei-LP-Book}.

While \eqref{eq_Kol}  forms a linear system of equations, analytical, or even numerical computations of $p(t;x)$ is often unwieldy even for bounded systems.  
In applications it is often of interest to know the moments $E(\|X(t)\|^r)$ for 
$r \in \nat$ where $\|.\|$ is some norm on $\real^N$. 
When the propensity functions are linear (or affine), it is possible to derive 
evolution equations for the moments which are closed. However, for nonlinear propensities it is not straightforward to even to decide if the system has finite moments let alone compute those.  

The time evolution of expected value of some function $h$ of the state, $E(h(X(t)))$, satisfies the so called Dynkin's formula
\[
\frac{d}{dt} E(h(X(t))) = \sum_{j=1}^M E\left[ \left(h(X(t)+\nu_j)-h(X(t)) \right) a_j(X(t)) \right].
\]
While it is tempting to use $h(x)=\|x\|^r$ to derive the time evolution of 
the $r$th moment $E(\|X(t)\|^r)$, care must be taken as the above equation 
may not hold for unbounded functions $h$. 
In this paper we derive with care some sufficient conditions for the 
moments $E(\|X(t)\|^r)$ to exist for all $r \in \nat$ and satisfy an exponential (in time) growth bound. We also provide a necessary condition for 
moments $E(\|X(t)\|^r)$ to exist for all $r$ and all $t \geq 0$. 

A set of sufficient conditions under which a large class of queueing networks 
(which are time inhomegeneous Markov processes on $\posint^N$) have moments 
converging asymptotically as $t \to \infty$ is obtained in 
\cite{Dai-Meyn-IEEEAUT95}. A recent work \cite{Briat-Gupta+Arxiv13} obtains 
a set of sufficient conditions under which $\sup_{t \geq 0} E(\|X(t)\|^r) < \infty$. The results obtained in this paper are for existence of moments for all finite $t \geq 0$ without requiring that $\sup_{t \geq 0} E(\|X(t)\|^r) < \infty$. This allows for systems which experience exponential growth (in time). 
Some sufficient conditions for the existence of moments for all finite $t \geq 0$ in the form of one-sided Lipschitz condition may be found for stochastic differential equations (SDEs) driven by Brownian motion in 
\cite{Higham-Mao+_SINUM02,Mao-Book-SDE98}. The class of processes studied in this paper are of a different form and consequently our results are of a different flavor.  

The rest of the paper is organized as follows. In Section \ref{sec-bdd} we 
develop some mathematical preliminaries and provide necessary and sufficient 
conditions for what we call the stoichiometric boundedness of species. 
The analysis in this section is purely deterministic. 
In Section \ref{sec-moms} we provide three main results, two sufficient conditions and a necessary condition for the existence of all moments for all time $t \geq 0$. We illustrate our results via examples where appropriate. 

\section{Preliminaries and boundedness of species}\label{sec-bdd}
A {\em chemical system} or a {\em system} is characterized by a {\em stoichiometric matrix} $\nu \in \integ^{N \times M}$ and a {\em propensity function} $a:\posint^N \to \real_+^M$. When necessary the propensity function may be extended to the domain $\integ^N$ to be zero outside $\posint^N$.  
Associated to a chemical system and an initial condition $x \in \posint^N$ is a Markov process $X(t)$ in continuous time with $X(0)=x$ (with probability $1$) 
as described in the introduction. We shall assume the process $X$ to have paths that are continuous from the right with left hand limits. 
 We assume that the process $X$ is carried by a probability space $(\Omega,\sF,\text{Prob})$.

We shall say that a propensity function is {\em proper} if it satisfies the condition that for all $x \in \posint^N$ if $x + \nu_j \notin \posint^N$ then $a_j(x)=0$. We note that properness is necessary and sufficient to ensure that the process $X$ remains in $\posint^N$ when started in $\posint^N$. 
We shall say that 
the propensity function is {\em regular} if it satisfies the condition that for all $x \in \posint^N$, and all $j=1,\dots,M$,  $a_j(x)=0$ if and only if $x + \nu_j \notin \posint^N$. We observe that regularity implies properness. 
Throughout the rest of the paper we shall assume properness. When regularity is  assumed, it will be stated explicitly.

Consider a system with $N$ species reacting through $M$ reaction channels. 
%Fix the initial state of the system to be deterministic $X(0)=x_0$. 
We define the {\em accessible set} of states $\sA_{x} \subset \posint^N$ given an initial state $x \in \posint^N$ 
by the condition that $y \in \sA_{x}$ if and only if there exists $t>0$ such that
\[
\text{Prob}(X(t)=y \, | \, X(0)=x) > 0.
\]
We observe that from standard Markov chain theory \cite{Bremaud-MarkovChains} the 
above definition is unchanged if the phrase ``there exists $t>0$''  is replaced by 
``for every $t>0$''. 
Further more $y \in \sA_x$ if and only if there exists a finite sequence 
$(j_1,j_2,\dots,j_n)$ of indices which take values in $\{1,\dots,M\}$ such 
that  for $y^{(l)}$  where $l=0,1,\dots,n$ defined  by
\[
y^{(l+1)} = y^{(l)} + \nu_l, \quad l=0,1,\dots,n-1,
\]
with $y^{(0)}=x$ it holds that 
\[
a_{j_l}(y^{(l-1)}) > 0,  \quad l =1,\dots,n.
\] 

It is convenient to define the {\em stoichiometricaly accessible set} 
$\sS_{x} \subset \posint^N$ given an initial state $x \in \posint^N$ by 
\begin{equation}\label{def_Sx}
\begin{aligned}
\sS_x &= \{ y \in \real^N \, | \,\exists v \in \posint^M \text{ such that } y = x + \nu v \} \cap \real_+^N \\
&= \{ y \in \posint^N \, | \,\exists v \in \posint^M \text{ such that } y = x+ \nu v \}.
\end{aligned}
\end{equation} 
(The second equality follows logically). 

It is clear that $\sA_{x} \subset \sS_{x}$. However these sets are not 
always equal as seen from Example 1.

{\bf Example 1:} Consider a system with $N=M=2$, $\nu_1=(3,-2)^T$ and 
$\nu_2=(-2,3)^T$. Consider the initial state $x=(1,1)^T$. 
Under the assumption of proper propensity function, at the initial state, 
propensities of both reactions are zero since the firing of either of the reactions will lead to a state with negative components. Thus $\sA_{x} = \{x\}$. 
However $\sS_{x}$ contains an infinite number of elements as choosing $k=(n,n)^T$ where $n$ is a positive integer 
results in $y = x + \nu k = (1+n,1+n)^T$ which are all in $\sS_{x}$ by definition.  

For $i=1,\dots,N$ let $\pi_i:\real^N \to \real$ be the standard projection onto the $i$th coordinate. Then if $\pi_i(\sA_{x})$ is bounded above we may conclude that 
species $i$ is bounded for initial condition $x$. Deciding whether $\pi_i(\sA_{x})$ is bounded above is harder than deciding whether $\pi_i(\sS_{x})$ 
is bounded above. So we shall focus on the latter first. We shall
use the terminology that species $i$ is {\em stoichiometricaly bounded} for the 
initial condition $x \in \posint^N$ provided $\pi_i(\sS_{x})$ is bounded above. As we shall see it turns out that stoichiometric boundedness is independent of the initial state $x$ and hence we could drop the reference to initial state when talking about stoichiometric boundedness of a species. 

In order to study the sets $\sS_x$ 
it is instructive to consider the related sets $C_x$ and $C_x^+$ defined as follows. Given $x \in \real^N$ 
we define $C_x$ and $C_x^+$ as follows:
\begin{equation}\label{def_Cx}
C_x = \{ y \in \real^N \, | \, \exists v \in \real_+^M \text{ such that } y = x + \nu v \},
\end{equation}
\begin{equation}\label{def_Cx+}
C_x^+ = C_x \cap \real_+^N.
\end{equation}
We note that $C_x$ is a {\em closed convex cone} with vertex $x$ and $C_x^+$ 
is a {\em closed convex set}.

\begin{lemma}\label{lem_integral}
Let $A \in \real^{k \times n}$ and $B \in \integ^{k \times n}$. 
Suppose there exists $v \in \real_+^n$ 
such that $A v > 0$ and $B v =0$. Then there exists $w \in \posint^n$  
 such that $A w > 0$ and $B w =0$.
\end{lemma}
\begin{proof}
Define
\[
P=\{u \in \real_+^n \, | \, A u >0 , \, B u = 0\}.
\] 
Note that $P$ is nonempty, a cone with vertex $0$ and is 
relatively open in $\ker(B)$. Since $B$ has integer entrees and $\ker(B) \cap \real_+^n$ is nonempty, it follows that $\ker(B) \cap \rat_+^n$ is nonempty. 
As a relatively open set in $\ker(B) \cap \real_+^n$, the set $P$ contains elements from $\rat_+^n$. Since $P$ is a cone with vertex $0$, by taking 
a suitable positive integer multiple we can conclude $P$ contains elements from $\posint^n$. 
\end{proof}
 
\begin{corollary}\label{corr-integral}
Let $A \in \real^{k \times n}$ and $B \in \integ^{k \times n}$. 
Suppose there exists $v \in \real_+^n$ 
such that $A v > 0$ and $B v \geq 0$. Then there exists $w \in \posint^n$  
 such that $A w > 0$ and $B w \geq 0$.
\end{corollary}
\begin{proof}
Follows from Lemma \ref{lem_integral}.
\end{proof}

\begin{lemma}\label{lem_Cxbdd}
Let $1 \leq i \leq N$ and $x \in \real_+^N$. Then $\pi_i(C_x^+)$ is bounded above if and only if for every $z \in C_x^+$ if $z_i > x_i$ then $z-x$ has 
at least one negative component.
\end{lemma}
\begin{proof}
{\em only if:}  Let $z \in  C_x^+$ and suppose $z_i>x_i$. If for all $\lambda>0$,
 the vector $x + \lambda(z-x)$ has no negative components then it would imply that $\pi_i(C_x^+)$ is unbounded. 
Thus there exists $\lambda>0$ such that $x + \lambda(z-x)$ has 
at least one negative component. This implies that $z-x$ has at least one negative component.

{\em if:} 
For any $z \in C_x^+ \setminus\{x\}$ define $L_z$ by  
\[
L_z = \{ y \in \real_+^N | \, \exists \lambda \geq 0 \text{ such that } y = x + \lambda (z-x) \},
\]
which will be a closed line segment (may be infinite). By assumption the set 
$\pi_i(L_z)$ is bounded above. To see this if $z_i \leq x_i$ then $\pi_i(L_z)$ is bounded above by $x_i$. If $z_i>x_i$ then $L_z$ is a finite segment since 
$z-x$ has at least one negative component. Since $C_x^+$ may be partitioned into sets of the form $L_z$, what is left to be shown is that there exists a common upper bound $M>0$ such that $\pi_i(L_z)$ is bounded above by $M$ for all $z \in C_x^+$. 

To see this define $f^x_i:C_x^+ \setminus\{x\} \to \real$ by  
\[
f^x_i(z) = \max\{y_i \, | y \in L_z \} = \max(\pi_i(L_z)),
\]
which is well defined. 
It is not difficult to show $f^x_i$ is continuous on $C_x^+ \setminus\{x\}$ 
and constant on the sets $L_z$, $z \in C_x^+ \setminus \{x\}$. Hence on the compact set 
\[
\{y \in \real^N \, |\, \|y-x\|=1\} \cap C_x^+
\]
$f^x_i$  attains a maximum value say $M$. It follows that $\pi_i(C_x^+)$ has maximum value $M$.
\end{proof}

\begin{corollary}\label{corr_Sbdd}
Let $1 \leq i \leq N$ and $x \in \posint^N$. Then $\pi_i(\sS_x)$ is bounded above if and only if for every $z \in \sS_x$ if $z_i > x_i$ then $z-x$ has 
at least one negative component.
\end{corollary}
\begin{proof}
{\em only if:} If $\pi_i(\sS_x)$ is bounded above then so is $\pi_i(C^+_x)$ and by Lemma \ref{lem_Cxbdd} the result follows. 

{\em if:} Suppose $\pi_i(\sS_x)$ is unbounded above. Then so is $\pi_i(C^+_x)$ and by Lemma \ref{lem_Cxbdd} there exists $y \in C_x^+$ such that $y_i>x_i$ and $y \geq x$. Hence there exists $v \in \real_+^M$ such that $\mu v >0$ and $\nu v \geq 0$ where $\mu$ is the $i$th row of $\nu$. By Corollary \ref{corr-integral} there exists $w \in \posint^M$ such that $\mu w >0$ and $\nu w \geq 0$. Taking $z = x + \nu w$ show that there exists $z \in \sS_x$ such that $z_i>x_i$ and $z \geq x$. 
\end{proof}

\begin{lemma}\label{lem_C0bddCxbdd}
Let $1 \leq i \leq N$ and $x \in \real_+^N$. Then $\pi_i(C_x^+)$ is bounded above if and only if $\pi_i(C_0^+)$ is bounded above. 
\end{lemma}
\begin{proof}
{\em only if:} 
We note that $C_x = \{x\} + C_0$ and hence 
\[
(C_0 \cap \real_+^N)  + \{x\} = C_x \cap (\real_+^N + \{x\}) \subset C_x \cap \real_+^N.
\]
Thus $\pi_i(C_0^+ + \{x\})$ is bounded above and hence so is $\pi_i(C_0^+)$.

{\em if:}  Suppose $\pi_i(C_x^+)$ is unbounded above. Then by Lemma \ref{lem_Cxbdd} there exists $z \in C_x^+$ such that $z_i>x_i$ and $z-x \geq 0$. This implies $z-x \in C_0^+$ and $(z-x)_i>0$, which in turn implies that $\pi_i(C_0^+)$ is unbounded above.
\end{proof}

For $x \in \posint^N$ the study of $\sS_x$ reduces to the study of $C_x^+$ because of the following lemma. 

\begin{lemma}\label{lem_SbddCbdd}
Let $1 \leq i \leq N$ and $x \in \posint^N$. Then $\pi_i(\sS_x)$ is bounded above if and only if $\pi_i(C_x^+)$ is bounded above. 
\end{lemma}
\begin{proof}
{\em if:} Follows since $\sS_x \subset C_x^+$. 
 
{\em only if:} Let $\mu \in \integ^M$ be such that $\mu^T$ is the $i$th 
row of $\nu$. If $\pi_i(C_x^+)$ is unbounded above then by Lemma \ref{lem_Cxbdd} there exists $z \in C_x^+$ such that $z_i > x_i$ and $z \geq x$. Hence there exists $v \in \real_+^M$ such that $\mu^T v > 0$ and $\nu v \geq 0$. 
From Lemma \ref{lem_integral} we may conclude that there exists $w \in \posint^M$ 
such that $\mu^T w > 0$ and $\nu w \geq 0$. It follows that the sequence $y^{(n)}$ 
defined by $y^{(n)} = x_i + n \mu^T w$ is a sequence in $\pi_i(\sS_x)$ that tends to $+\infty$.
\end{proof}

\begin{corollary}\label{corr_SbddC0bdd}
Let $1 \leq i \leq N$ and $x \in \posint^N$. Then $\pi_i(\sS_x)$ is bounded above if and only if $\pi_i(C_0^+)$ is bounded above. 
\end{corollary}
\begin{proof}
Follows from Lemmas \ref{lem_C0bddCxbdd} and \ref{lem_SbddCbdd}. 
\end{proof}

Suppose a certain non-negative linear combination 
\[
\alpha_1 X_1(t) + \dots + \alpha_N X_N(t)
\]
of species is always nonincreasing with time and suppose $\alpha_i>0$.
Then we can write 
\[
X_i(t) \leq (1/\alpha_i) \sum_{j \neq i} \alpha_j X_j(0),
\]  
to conclude that species $i$ is bounded. The existence of a nondecreasing 
non-negative linear combination can be equivalently  stated as the existence of 
$\alpha \geq 0$ such that $\alpha^T \nu \leq 0$.
 
However the fact that the converse is also true is not obvious and 
requires results from the study of convex and cone sets as seen in the 
following theorem which 
provides a necessary and sufficient condition for stoichiometric boundedness 
of a species. 

\begin{theorem}
Species $i$ is stoichiometricaly 
bounded if and only if there exists a vector $\alpha \in \posint^N$ such that $\alpha \geq 0$, 
$\alpha_i>0$ and $\alpha^T \nu \leq 0$. 
\label{thm_stoichbdd}
\end{theorem}
\begin{proof}
We shall use Corollary \ref{corr_SbddC0bdd} to work with $C_0^+$. 

{\em if}: For all $y \in C_0$ there exists $v \in \real^M$ such that $v \geq 0$ and $y = \nu v$. Now suppose $y \in C_0^+$. Then $\alpha^T y \geq 0$.
However, since $\alpha^T \nu \leq 0$, and $v \geq 0$, we have that
\[
\alpha^T y = \alpha^T \nu v \leq 0.
\]
Thus $\alpha^T y =0$. Since $\alpha_i>0$ it follows that $y_i=0$. 
Hence $\pi_i(C_0^+)=\{0\}$, and is bounded above.

{\em only if}: Define the set $B_0$ by 
\[
B_0 = \{ \alpha \in \real^N \, | \, \alpha^T \nu \leq 0 \},
\]
and note that $B_0=(C_0)^o$ i.e. the polar of $C_0$. (See Appendix for some basics on convex analysis and definitions).
To see this suppose $\alpha \in B_0$. If $y \in C_0$ then there exists $v \geq 0$ such that $y = \nu v$ and hence $\alpha^T y = \alpha \nu v \leq 0$. Hence $\alpha \in (C_0)^o$. Conversely if $\alpha \in (C_0)^o$ then for all $y \in C_0$ it holds that $\alpha^T y \leq 0$. Since $\nu_1,\dots,\nu_M \in C_0$ it follows that $\alpha^T \nu \leq 0$ and thus $\alpha \in B_0$. 

Since $\pi_i(C_0^+)$ is bounded above it follows that $y_i=0$ for all $y \in C_0^+$. To see this, suppose $y \in C_0^+$ and $y_i>0$. Since $C_0^+$ is a cone by taking positive multiples of $y$ we can obtain arbitrarily large elements in $\pi_i(C_0^+)$ violating the assumption that $\pi_i(C_0^+)$ is bounded above.
Hence $e_i \in \real^N$ ($i$th standard basis vector) satisfies $e_i^T y =0 \leq 0$ for all $y \in C_0^+$ and therefore by definition $e_i \in (C_0^+)^o$. 

Now using $(\real_+^N)^o = -\real_+^N$, $B_0=(C_0)^o$ and Lemma \ref{lem_polar_cap} we have that
\[
(C_0^+)^o = (C_0 \cap \real_+^N)^o = \text{cl}((C_0)^o + (\real_+^N)^o) = \text{cl}(B_0 - \real_+^N).
\]
Since $B_0$ and $\real_+^N$ are polyhedral so is $B_0 - \real_+^N$ and hence
 $B_0 - \real_+^N$ is closed. Thus $e_i \in B_0 - \real_+^N$. So $e_i = \alpha - u$ for some $\alpha \in B_0$ and $u \in \real_+^N$. Hence $\alpha = u + e_i$ and thus $\alpha \geq 0$ and $\alpha_i>0$. Since $\alpha \in B_0$ it follows that $\alpha^T \nu \leq 0$.  

Thus we have shown that there exists $\alpha \in \real_+^N$ such that $\alpha_i>0$ and $\alpha^T \nu \leq 0$. By Corollary \ref{corr-integral} it follows that we can choose such $\alpha \in \posint^N$.

\end{proof}

The following corollary is immediate. 
\begin{corollary}
A subset $I \subset \{1,\dots,N\}$ of species is stoichiometricaly 
bounded if and only if there exists a vector $\alpha \in \posint^N$ such that $\alpha \geq 0$, 
$\alpha_i>0$ for $i \in I$ and $\alpha^T \nu \leq 0$. 
\label{cor_stoichbdd}
\end{corollary}

In order to facilitate the discussion of boundedness of species we shall define the notion of 
a {\em counting sequence} as follows. A finite sequence $(u_1,\dots,u_m)$ where $u_j \in \posint^M$ is said to be a {\em counting sequence} provided $u_1=0$ and for $j=1,\dots,m-1$ $u_{j+1}-u_j$ has precisely one component 
of value $1$ with all other components being zero. The following lemma is immediate.

\begin{lemma}\label{lem_count}
Suppose the propensity function is regular. Then for given $x \in \posint^N$, a state $y \in \sA_x$ if and only if there exists a
counting sequence $(u_1,\dots,u_m)$ in $\posint^M$ such that 
\[
x + \nu u_j \geq 0, \quad j=1,\dots,m,
\]
and $y = x + \nu u_m$. 
\end{lemma}
\begin{proof}
If $y \in \sA_x$ there must be a sequence of reaction events which can move 
the state from $x$ to $y$ without leaving $\posint^N$. Conversely if there 
is such a sequence then under the assumption of regularity of the propensity function, such a sequence will have nonzero probability of happening. 
\end{proof}

Finally we have the following theorem which relates boundedness of a species with its stochiometric boundedness. 

\begin{theorem}\label{thm_bdd_sbdd}
\begin{enumerate}
\item If species $i$ is stoichiometricaly bounded then it is bounded. 
\item Conversely 
if species $i$ is stoichiometricaly unbounded and the propensity function 
is regular then the species $i$  is unbounded for all 
sufficiently large initial conditions.   
\end{enumerate}
\end{theorem}
\begin{proof}
The first part is obvious. We shall prove the second part. 
Since the species $i$ is stoichiometricaly unbounded, from Corollary \ref{corr_Sbdd} there exists 
$v \in \posint^M$ such that $\mu^T v >0$ and $\nu v \geq 0$ where 
$\mu^T$ is the $i$th row of $\nu$. 
%Thus the set $V$ defined by
%\[
%V = \{u \in \posint^M \, | \, \mu^T u >0, \, \nu u \geq 0\}
%\]
%is nonempty. Let $v_0$ be an element in $V$, and 
Let $u_1,u_2,\dots,u_m$ be a counting sequence with $u_m=v$ and define $\bar{x} \in \posint^N$ by the condition that for $i=1,\dots,N$ 
\[
\bar{x}_i = \max\{0, -(\nu u_1)_i, -(\nu u_2)_i, \dots, -(\nu u_m)_i \}.
\]
Then for all $x \in \posint^N$, satisfying $x \geq \bar{x}$ we have that 
\[
x + \nu u_j \geq 0, \quad j=1,\dots,m.
\]
Thus it follows by Lemma \ref{lem_count} that $x + \nu v \in \sA_x$. 
Define the sequence $(y^{(n)})$ for $n \in \nat$ by $y^{(n)} = x + n \nu v$. 
It is easy to show using mathematical induction that $y^{(n)} \in \sA_x$ for all $n$ and that $y^{(n)}_i$ 
is strictly increasing so that 
$\pi_i(\sA_x)$ is unbounded above. 
 
\end{proof}

\section{Moment growth bounds}\label{sec-moms}
In order to facilitate the development of results concerning moment growth bounds we shall define {\em critical species} and {\em critical reactions} as follows. 

We say that species $i$ is a critical species if and only if it is not 
stoichiometricaly bounded. Without loss of generality we assume that the 
species are ordered such that the copy number vector 
$x=(y,z) \in \posint^{N_c} \times \posint^{N-N_c}$ where $y$ is the 
copy number vector of critical species,  $z$ is the copy number vector of 
non-critical species and $N_c$ is the number of critical species. 
A reaction channel $j$ is non-critical if and only if there 
exists $H:\posint^{N-N_c} \to \real$ such that 
\[
a_j(x) \leq H(z) (\|y \|+1), \quad \forall x = (y,z) \in \posint^{N}.
\]
In other words critical reactions are those whose propensities grow faster than linearly in the critical species. We shall use $M_c$ to denote the number of critical reactions. 
Without loss of generality we shall assume that the reaction channels are ordered so that $j=1,\dots,M_c$ correspond to the critical reactions. 

In what follows, 
given a system with stochiometric matrix $\nu$, we define the $N_c \times M$ matrix $\nu^1$ termed the {\em critical species stoichiometric matrix} to be 
the submatrix of $\nu$ consisting of the rows $1,\dots,N_c$ corresponding to 
the critical species and we define the $(N-N_c) \times M$ matrix 
$\nu^2$ termed the {\em non-critical species stoichiometric matrix} to be the 
submatrix of $\nu$ which consists of rows $N_c+1,\dots,N$ corresponding to 
non-critical species. We also define the $N_c \times M_c$ matrix $\nu^c$ termed the {\em critical stoichiometric matrix} to be the 
submatrix of $\nu$ consisting of the rows $1,\dots,N_c$ corresponding to 
the critical species and columns $1,\dots,M_c$ corresponding to critical reactions. 

We first state a lemma. 
\begin{lemma}\label{lem_convenient}
Suppose a system has regular propensity functions. Let $J \subset \{1,\dots,M\}$ be a subset of reactions. 
Then the following are equivalent:
\begin{enumerate}
\item There exists a norm $\| .\|$ in $\real^N$ such that the following holds:
for all $x \in \posint^N$ and for all $j \in J$ 
\[
 x+\nu_j \in \posint^N  \Rightarrow   \|x+\nu_j\| \leq \|x\|.
\]
\item  For the system consisting only of the reactions in $J$ all the species are stoichiometricaly bounded.
\item There exists $\alpha \in \posint^N$ such that $\alpha>0$ and $\alpha^T \nu_j \leq 0$ for all $j \in J$.
\end{enumerate}
\end{lemma}

\begin{proof}
First we note that conditions 2 and 3 are equivalent by Corollary \ref{cor_stoichbdd}. It is also clear that 1 implies 2 (and hence 3). Thus it suffices to show 3 implies 1. Suppose 3 holds. Define the norm $\|.\|$ on $\real^N$ by 
\[
\|x\| = \sum_{i=1}^N \alpha_i |x_i|.
\]
Let $x \in \posint^N$ and suppose $x+\nu_j \in \posint^N$ for some 
$j \in J$. Then 
\[
\|x+\nu_j\| = \sum_{i=1}^N \alpha_i (x_i+\nu_{ij}) = \sum_{i=1}^N \alpha_i x_i + \sum_{i=1}^N \alpha_i \nu_{ij} = \|x\| + \alpha^T \nu_j \leq \|x\|.
\]
\end{proof}

We remark that Lemma \ref{lem_convenient} is typically used with 
$J=\{1,\dots,M_c\}$, the set of critical of reactions.

In order to discuss how moments $E(\|X(t)\|^r)$ evolve in time, 
first we note that the {\em generator} $\sA$ of the Markov process with stoichiometric matrix $\nu$ and propensity function $a$ is given by
\begin{equation}\label{eq_gen}
(\sA h)(x) = \sum_{j=1}^M \left(h(x+\nu_j)-h(x) \right) a_j(x),
\end{equation}
where $h:\posint^N \to \real$ and we use the convention that $h(y)=0$ if $y \notin \posint^N$. We note that $\sA$ is regarded as an operator on the Banach 
space $\sL$ of bounded functions $h:\posint^N \to \real$ and that the domain of $\sA$ is not all of $\sL$ as $a_j$ are typically not bounded functions. 
However the collection of all functions $h$ that are constant outside a compact subset of $\posint^N$ are in the domain of $\sA$.

  It follows from standard Markov process theory that for all functions $h:\posint^N \to \real$ in the domain of $\sA$ the following formula, some times known as Dynkin's formula, holds for all $t \geq 0$:
\begin{equation}\label{eq_Dynkin}
\frac{d}{dt} E(h(X(t))) = \sum_{j=1}^M E\left[ \left(h(X(t)+\nu_j)-h(X(t)) \right) a_j(X(t)) \right],
\end{equation} 
or equivalently in integral form
\begin{equation}\label{eq_Dynkin_int}
E(h(X(t))) = E(h(X(0))) + \sum_{j=1}^M \int_0^t E\left[ \left(h(X(s)+\nu_j)-h(X(s)) \right) a_j(X(s)) \right] ds.
\end{equation} 
We suggest \cite{EK-book} as a general reference. 

For $r \in \nat$ we define the class $\scP_r$ to be the set of functions $f:\posint^N \to \real$ 
characterized by the condition that $f \in \scP_r$ if and only if there exist
$H>0$ such that 
\[
|f(x)| \leq H (\|x\|^r +1), \quad \forall x \in \posint^N,
\]
and define $\scP_r^+$ to denote the subset of $\scP_r$ consisting of non-negative functions. We also define $\scP$ by 
\[
\scP = \bigcup_{r \in \posint} \scP_r,
\]
and $\scP^+$ to denote the subset of $\scP$ consisting of non-negative functions.
We observe that the definition of classes  $\scP_r, \scP$ is independent of the choice of norm on $\real^N$. We establish a few lemmas about classes $\scP_r,\scP$ first. 

\begin{lemma}\label{lem_r_s}
Suppose $r,s \in \posint$ and $r \leq s$. Then there exists $H>0$ such that 
\[
\|x\|^r \leq H \|x\|^s, \quad x \in \posint^N.
\]
Thus $\scP_r \subset \scP_s$. 
\end{lemma}
\begin{proof}
This clearly holds for $x=0$ and for $x$ that satisfy $\|x\| \geq 1$ it holds with $H=1$. Since the set of $x$ for which $\|x\|<1$ is finite one may find $H$ large enough for this to hold for all $x \in \posint^N$.
\end{proof} 

\begin{lemma}\label{lem_classP}
The classes $\scP_r,\scP$ are vector spaces (over $\real$) and any (multivariate) polynomial belongs to class $\scP$. 
Suppose $f \in \scP_r$, $y \in \integ^N$ and $g:\posint^N \to \real$ is defined by $g(x) = f(x+y)$ if $x+y \in \posint^N$ else $g(x)=0$. Then $g \in \scP_r$. In other words $\scP_r$ (and hence $\scP$) are shift invariant.
Finally if $f \in \scP_r$ and $g \in \scP_s$ then $h \in \scP_{s+r}$ where $h=fg$.   
\end{lemma}
\begin{proof}
It is trivial to see that $\scP_r$ is a vector space. Given $f,g \in \scP$ 
by Lemma \ref{lem_r_s} there exists some $r \in \posint$ such that $f,g \in \scP_r$. Hence it is clear then that $\scP$ is a vector space as well. 

In order to show that all polynomials belong to $\scP$ it is adequate to show that all monomials $p(x)=x^\beta$ where $\beta=(\beta_1,\dots,\beta_N) \in \posint^N$ and 
\[
x^\beta = x_1^{\beta_1} \dots x_N^{\beta_N},
\]
belong to $\scP$. Indeed
\[
|x^\beta| \leq (\|x\|_{\infty}^{\beta_0}+1),\quad \forall x \in \posint^N,
\]
where $\beta_0=\max\{\beta_1,\dots,\beta_N\}$. Using equivalence of norms there exists $K$ independent of $x$ such that
\[
|x^\beta| \leq K (\|x\|^{\beta_0} + 1), \quad \forall x \in \posint^N.
\] 

To show shift invariance it is adequate to note that for $r \in \posint$ and $x,y \in \integ^N$
\[
\|x+y\|^r \leq (\|x\| + \|y\|)^r \leq \sum_{l=0}^r \frac{r!}{l!(r-l)!} \|x\|^l \|y\|^{r-l} \leq K_r (\|x\|^r +1),
\]
where $K_r$ depends on $y$, $r$ and is obtained in part by Lemma \ref{lem_r_s}. 
Finally if $f \in \scP_r$ and $g \in \scP_s$ and $h=fg$ then for some $H>0$ 
and some $H'>0$ independent of $x$ we have
\[
|f(x)g(x)| \leq H (\|x\|^r + 1)(\|x\|^s + 1) \leq H' (\|x\|^{r+s} +1),
\]
where we have used Lemma \ref{lem_r_s}.
\end{proof}

Equations \eqref{eq_Dynkin} and \eqref{eq_Dynkin_int} hold for $h$ that are constant outside a compact set. The following lemma shows that under suitable assumptions these equations hold for all $h \in \scP$. 

\begin{lemma}\label{lem_mom_smooth}
Let $r \in \nat$ and suppose that $E(\|X(t)\|^r)<\infty$ for all $t \geq 0$.
Then for every $f \in \scP_r$, $E(|f(X(t))|)<\infty$ 
for every $t \geq 0$ and $E(f(X(t)))$ is continuous in $t$ 
for $t \geq 0$. 

Suppose in addition that the propensity functions $a_j$ for $j=1,\dots,M$ all 
belong to class $\scP_{s}$ where $1 \leq s \leq r$. Then for each $f \in \scP_{r-s}$, $E(f(X(t)))$ is continuously differentiable in $t$ for $t \geq 0$ and \eqref{eq_Dynkin}, \eqref{eq_Dynkin_int} hold with $h=f$.
\end{lemma}
\begin{proof}  
Given $f \in \scP_r$, the claim $E(|f(X(t))|)<\infty$ is obvious under the assumptions. 

To show $E(f(X(t)))$ is continuous in $t$ we first consider $f \in \scP_r^+$. For each $K>0$ define $f^K:\posint^N \to \real$ 
by $f^K(x) = f(x) \wedge K$, where $a \wedge b$ denotes the minimum of $a$ and $b$.
 Since for each $K$, $f^K$ is constant outside a compact set, Dynkin's 
formula \eqref{eq_Dynkin} (with $h=f^K$) holds showing $E(f^K(X(t)))$ to 
be differentiable and hence continuous in $t$. Since $f^K \uparrow f$ as $K \uparrow \infty$, by monotone convergence $E(f^K(X(t))) \uparrow E(f(X(t)))$ as $K \uparrow \infty$. 
Hence Dini's theorem and a standard argument show that 
$E(f(X(t)))$ is continuous in $t$ for $t \geq 0$. For $f \in \scP_r$ the proof is completed by decomposing $f$ into its positive and negative parts, $f = f^+ - f^-$.  
Thus we have established that for every $f \in \scP_r$, 
$E(f(X(t)))$ is continuous in $t$ for $t \geq 0$. 

To show the second part
% $E(f(X(t)))$ is continuously differentiable in $t$ 
%under the assumption that $a_j \in \scP$ for $j=1,\dots,M$, 
we consider $f \in \scP_{r-s}^+$ and for $K>0$ we consider  
the integral equation \eqref{eq_Dynkin_int} with $h=f^K$. 
We observe that since $a_j \in \scP_s$, if $f \in \scP_{r-s}$ then $\sA (f^K) \in \scP_r$, $\sA f \in \scP_r$ and that
 $f^K(X(t,\omega)) \to f(X(t,\omega))$ for almost all $(t,\omega)$ as 
$K \uparrow \infty$ where the Lebesgue measure is used for $t \geq 0$.
 Next we bound $\sA (f^K)$ as
\[
|(\sA f^K)(x)| \leq \sum_{j=1}^M f^K(x+\nu_j) a_j(x) + \sum_{j=1}^M f^K(x) a_j(x) \leq g(x),
\]
where
\[
g(x) = \sum_{j=1}^M f(x+\nu_j) a_j(x) + \sum_{j=1}^M f(x) a_j(x), 
\]
and we also observe that $g \in \scP_r^+$. Thus $E(g(X(t)))$ is finite and continuous in $t$ and thus $\int_0^t E(g(X(s))) ds < \infty$ for each $t \geq 0$. 
Hence the dominated convergence theorem allows us to conclude that one could take the limit as $K \to \infty$ on both sides of \eqref{eq_Dynkin_int} with $h=f^K$ 
to conclude that the equation holds for $h=f \in \scP_{r-s}^+$ with all terms being finite. This shows $E(f(X(t)))$ to be continuously differentiable in $t$ for $f \in \scP_{r-s}^+$. The proof is completed for $f \in \scP_{r-s}$ by decomposing $f$ into positive and negative parts.

\end{proof}

\begin{lemma}\label{lem-growth-phi}
Suppose $\phi:[0,\infty) \to \real$ is strictly positive for all $t \geq 0$, 
differentiable at $0$ and suppose there exist $H>0$ and $\lambda \in \real$ such that for all $t \geq 0$,
\[
\phi(t) \leq H \phi(0) e^{\lambda t}.
\]
Then there exists $\mu \in \real$ such that for all $t \geq 0$,
\[
\phi(t) \leq \phi(0) e^{\mu t}.
\]
\end{lemma}
\begin{proof}
We observe that for $t > 0$,
\[
\frac{\ln(\phi(t)) - \ln(\phi(0))}{t} \leq \frac{\ln(H)}{t} + \lambda.
\] 
We note that the right hand side is bounded for $t \geq t_0$ for every $t_0>0$,
and the left hand side is bounded for $t \in (0,t_0]$ for some $t_0>0$ since the limit as $t \to 0+$ exists and is finite by assumption. Hence the left hand side is bounded for $t \in (0,\infty)$. We set
\[
\mu = \sup_{t > 0} \Big\{\frac{\ln(\phi(t)) - \ln(\phi(0))}{t}\Big\} < \infty,
\]
to obtain the result.
\end{proof}

The following theorem provides a sufficient condition for exponential moment growth bounds. 

\begin{theorem}\label{thm1_mom_bnd}
Let $\nu^c$ be defined as above and suppose propensity functions belong to class $\scP$. 
Suppose further that there exists $\alpha \in \posint^{N_c}$ such that $\alpha>0$ and 
$\alpha^T \nu^c \leq 0$. Then for each $r \in \nat$ there 
exists $\mu_r$ such that the following holds for all $t \geq 0$ and in any norm $\|.\|$ on $\real^N$:
\begin{equation}\label{eqn_mom_bnd}
E(\|X(t)\|^r) \leq E(\|X(0)\|^r) e^{\mu_r t} + e^{\mu_r t} - 1.
\end{equation}
\end{theorem}
\begin{proof}
First we claim that there exists $\gamma = (\gamma_1,\gamma_2) \in \posint^N$ 
where $\gamma_1 \in \posint^{N_c}$ and $\gamma_2 \in \posint^{N-N_c}$ such that
$\gamma>0$, $\gamma_1^T \nu^1_j \leq 0$ for $j=1,\dots,M_c$ and $\gamma_2^T \nu^2_j \leq 0$  for $j=1,\dots,M$ where $\nu^1,\nu^2$ are as defined earlier. To see this we first observe that by Corollary \ref{cor_stoichbdd} 
there exists $\beta=(\beta_1,\beta_2) \in \posint^N$ where $\beta_1 \in \posint^{N_c}$ and $\beta_2 \in \posint^{N-N_c}$ such that $\beta_1^T \nu^1 \leq 0$, 
$\beta_2^T \nu^2 \leq 0$ and $\beta_2 >0$. Set $\gamma_1 = \beta_1 + \alpha$ 
and $\gamma_2 = \beta_2$ to obtain the desired result. 

We shall use the norm defined by 
$\|x\| = \gamma^T |x|$ (where $|x| = (|x_1|,|x_2|,\dots,|x_N|)$) and for $x \geq 0$ we have that $\|x\| = \gamma^T x$. 
 For $r \in \nat$ and $K>0$ define $f_r,f_r^K:\posint^N \to \real$ by
\[
f_r(x) = \|x\|^r,  \quad f^K_r(x) = \|x\|^r \wedge K.
\]
%where $a \wedge b$ stands for the minimum of $a$ and $b$.
%where $x=(y,z)$ with $y \in \posint^{N_c}$ denoting the vector of critical species counts. 
We shall define $f_r,f_r^K$ to be zero outside $\posint^N$. 
It follows that $f^K_r$ are constant outside a compact set  for each $K>0$ (and hence in the domain of the generator $\sA$) and $f^K_r \uparrow f_r$ as $K \uparrow \infty$. 

We write
$(\sA f^K_r)(x) = \sum_{j=1}^M T_j$
where 
\[
T_j = [f^K_r(x+\nu_j) - f^K_r(x)]a_j(x).
\] 
We note that $a_j(x)=0$ and hence $T_j=0$, unless $x \in \posint^N$ and $x+\nu_j \in \posint^N$.
Since we seek a non-negative upper bound for $T_j$ we shall only consider the case when $x \in \posint^N$ and $x+\nu_j \in \posint^N$.

When $j=1,\dots,M_c$, due to the choice of our norm, we obtain that for $x=(y,z)$ 
\[
\|x + \nu_j\| = \gamma_1^T y + \gamma_2^T z + \gamma_1^T \nu^1_j + \gamma_2^T \nu^2_j \leq \gamma_1^T y + \gamma_2^T z = \|x\|.
\]
Given this, we obtain that $T_j \leq 0$ regardless of the value of $K$. 

To bound $T_j$ for $j=M_c+1,\dots,M$, we consider the ordering of the three terms $\|x+\nu_j\|^r, \|x\|^r$ and $K$. If $\|x\|^r > K$ then regardless of the value of $\|x+\nu_j\|$ we obtain that
\[
T_j \leq 0.
\]
If $\|x\|^r \leq K$ then regardless of the value of $\|x+\nu_j\|^r$ we obtain that
\[
\begin{aligned}
T_j &\leq [ \|x+\nu_j\|^r - \|x\|^r] a_j(x) 
   = [(\|x\| + \gamma_1^T \nu^1_j + \gamma_2^T \nu^2_j)^r - \|x\|^r] a_j(x)\\
&\leq  [(\|x\| + \gamma_1^T \nu^1_j)^r - \|x\|^r] a_j(x)
 \leq H(z) \left( \sum_{l=0}^{r-1} \frac{r!}{l!(r-l)!} \|x\|^l (\gamma_1^T \nu^1_j)^{r-l} \right) (\|y\|+1)\\
&\leq H(z) \lambda'_r(\|x\|^r+1) = H(z) \lambda'_r(\|x\|^r \wedge K +1),
\end{aligned}
\]
where $\lambda'_r$ is a constant that does not depend on $x$ or $K$ and the Lemma \ref{lem_r_s} has been used. 
On account of positivity of the above upper bound, it provides an upper bound for $T_j$ when $j=M_c+1,\dots,M$ regardless of whether $\|x\|^r \leq K$ or not. 

Thus we obtain the bound 
\[
(\sA f^K_r)(x) \leq (M-M_c) \lambda'_r H(z) f_r^K(x) + (M-M_c) \lambda'_r H(z).
\]
Hence we obtain
\[
\frac{d}{dt} E(f_r^K(X(t))) \leq (M-M_c) \lambda'_r E[H(Z(t)) f_r^K(X(t))] + (M-M_c) \lambda'_r E[H(Z(t))].
\]
Using the fact that the vector copy number $Z(t)$ of the non-critical species is bounded, we obtain that
\[
\frac{d}{dt} E(f_r^K(X(t))) \leq \lambda_r E(f_r^K(X(t))) + \lambda_r,
\]
where $\lambda_r$ is another constant. 
The Gronwall Lemma yields that
\[
E(f_r^K(X(t))) \leq E(f_r^K(X(0))) e^{\lambda_r t} + e^{\lambda_r t} -1, \quad t \geq 0.
\]
Taking limit as $K \uparrow \infty$ and using the monotone convergence theorem we obtain
\[
E(f_r(X(t))) \leq E(f_r(X(0))) e^{\lambda_r t} + e^{\lambda_r t} -1, \quad t \geq 0.
\]
Hence in the specific norm $\|x\|=\gamma^T |x|$ we obtain
\[
E(\|X(t)\|^r) \leq E(\|X(0)\|^r) e^{\lambda_r t} + e^{\lambda_r t} -1, \quad t \geq 0.
\]
Using the equivalence of norms in $\real^N$ we obtain the bound (in any given norm)
\[
 E(\|X(t)\|^r) \leq L_r \, E(\|X(0)\|^r) e^{\lambda_r t} + L_r \, e^{\lambda_r t} - L_r, \quad t \geq 0,
\] 
where $L_r$ is a constant that depends only on the norm used and on $r$ and
%If $E(\|X(t)\|^r)=0$ for some $t>0$ then it by Lemma \ref{lem-exp-zero} the result follows. Otherwise $E(\|X(t)\|^r)\neq 0$ for all $t >0$. We note that 
considering $t=0$ it is clear that $L_r \geq 1$. Define $\phi(t)$ by 
\[
\phi(t) = E(\|X(t)\|^r) + 1, \quad t \geq 0.
\]
Then $\phi(t) >0$ for $t \geq 0$ and it follows that 
\[
\phi(t) \leq L_r \phi(0) e^{\lambda_r t} + (1 - L_r) \leq L_r \phi(0) e^{\lambda_r t}.
\]
By Lemma \ref{lem_mom_smooth} it is clear that $\phi$ is continuously differentiable in $t$ for $t \geq 0$. Lemma \ref{lem-growth-phi} clinches the desired result.

\end{proof}

{\bf Example 2} Consider the system with two species and two reactions 
given by $\nu_1 = (2,-1)^T$, $\nu_2=(-1,1)^T$ 
and $a_1(x) = x_2^2$ and $a_2(x) = x_1$. Since $2 \nu_1 + 3 \nu_2 = (1,1)^T$ 
it is clear that both species are critical (as they are stoichiometricaly unbounded). However only reaction $1$ is critical. Thus the critical stoichiometric matrix is the column vector $\nu_1$. The choice of $\gamma=(1,3)^T$ satisfies $\gamma^T \nu_1 =-1 <0$ and hence we can conclude that the moments of all orders exist and satisfy the 
exponential in time growth bound.  

The conditions of Theorem \ref{thm1_mom_bnd} are not necessary to ensure that a moment growth bound of the form \eqref{eqn_mom_bnd} holds.  

{\bf Example 3} Consider a birth/death process with birth rate $a_1(x) = x^m$ and death rate $a_2(x)=2x^m$. Then $\nu_1=1$ and $\nu_2=-1$ and the critical 
matrix $\nu^c = (1, -1)$. The conditions of Theorem \ref{thm1_mom_bnd} are not met if $m >1$.  Nevertheless, intuitively one expects the birth rate to be 
compensated by the death rate of the same form but of a dominant magnitude. If we set $f_r(x) = x^r$ then
\[
\begin{aligned}
(\sA f_r)(x) &= \left((x+1)^r - x^r) x^m\right) + 2 \left((x-1)^r - x^r) x^m\right)\\
&=  \sum_{l=0}^{r-1} \frac{r!}{l! (r-l)!} x^{l+m} + 2 \sum_{l=0}^{r-1} \frac{r!}{l! (r-l)!} (-1)^{r-l} x^{l+m}\\
&= (-r x^{r+m-1} + \frac{3 r(r-1)}{2} x^{r+m-2} + \dots). 
\end{aligned}
\]
When $m=2$ (quadratic birth/death rates) the positive term with highest power of $x$ is $3 r(r-1) x^{r}/2 $ and suitable truncation and Gronwall Lemma 
may be used to obtain an exponential growth bound on all moments. If $m>2$ then  the positive term  $3r(r-1) x^{r+m-2}/2$ is a higher power than $x^r$ 
and unless $r=1$ (in which case finiteness can be shown easily regardless of $m$) this approach does not work. 

Thus the intuition suggested above may only be valid if the propensities are quadratic at most. 

{\bf Example 4} Consider a two species ($S_1$ and $S_2$) model where when one $S_1$ and one $S_2$ come together one of three things can happen; the birth of an $S_1$, the birth of an $S_2$ or the death of both $S_1$ and $S_2$. This may be depicted by   
\[
S_1 + S_2 \to 2 S_1 + S_2, \quad S_1+ S_2 \to S_1 + 2 S_2, \quad S_1 + S_2 \to 0.
\]
Thus we have $\nu_1=(1,0)^T, \nu_2=(0,1)^T$ and $\nu_3=(-1,-1)^T$. It is easy to see that both species are critical. 

Suppose the propensities $a_1,a_2,a_3$ for these reactions are given by
\[
a_1(x) = a_2(x) = x_1 x_2, \quad a_3(x) = 2 x_1 x_2.
\]
Hence all three reactions are critical. 
Here again the conditions of Theorem \ref{thm1_mom_bnd} are not met. However the fact that one occurrence of the third reaction undoes one occurrence of both of the other two and the dominant rate of the third reaction might suggest the possibility of bounded moments. 

Let us use the $1$-norm and set $f_r(x) = \|x\|^r$. We obtain that
\[
(\sA f_r)(x) = [(y+1)^r - y^r] a_1(x) + [(y+1)^r - y^r] a_2(x) + [(y-2)^r - y^r] a_3(x),
\]
where $y=x_1+x_2$ and we suppose $x=(x_1,x_2) \geq 0$. Setting $a_1=a_2=a$ and $a_3=2a$ and simplifying we obtain
\[
\begin{aligned}
(\sA f_r)(x) &=  2a(x) \sum_{l=0}^{r-1} \frac{r!}{l! (r-l)!} y^l + 2a(x) \sum_{l=0}^{r-1} \frac{r!}{l! (r-l)!} (-2)^{r-l} y^l\\
&= (-r y^{r-1} a(x) + 5 r(r-1) y^{r-2} a(x) + \dots). 
\end{aligned}
\]
Since $a(x) = x_1 x_2 \leq \|x\|^2 = y^2$, similar to the $m=2$ case in Example 2, we may expect to obtain exponential growth bounds on all moments. However if $a(x)$ did not satisfy quadratic growth bound such bounds may not hold. 

The two examples above are examples of application of the following theorem.

\begin{theorem}
\label{thm2} Suppose that propensity functions all belong to class $\scP$ and
that there exist $C>0$ and $\gamma \in \real^N$ such that $\gamma>0$ and
\[
\gamma^T F(x) \leq  C (\|x\|+1), \quad \forall x \in \posint^N,
\]
where $F(x) = \sum_{j=1}^M \nu_j a_j(x)$. 
Further suppose there exists $H>0$ such that for all $j$ with $\gamma^T \nu_j \neq 0$, 
\[
a_j(x) \leq H (\|x\|^2 + 1).
\]
Then for each $r \in \nat$ there 
exists $\mu_r$ such that equation \eqref{eqn_mom_bnd} holds.   

\end{theorem}
\begin{proof}
Define $f_r, f^K_r$ as in the proof of Theorem \ref{thm1_mom_bnd}. 
We write
$(\sA f^K_r)(x) = \sum_{j=1}^M T_j$
where 
\[
T_j = [f^K_r(x+\nu_j) - f^K_r(x)]a_j(x).
\] 
We choose the norm defined by $\|x\| = \gamma^T |x|$ where $|x| = (|x_1|,|x_2|,\dots,|x_N|)$. When $\|x\|^r > K$ we have that $T_j \leq 0$.  
When $\|x\|^r \leq K$, we obtain 
\[
\begin{aligned}
T_j &\leq  [(\gamma^T x + \gamma^T \nu_j)^r - (\gamma^T x)^r] a_j(x)\\
&= r (\gamma^T x)^{r-1} \gamma^T \nu_j a_j(x) + \sum_{l=0}^{r-2} \frac{r!}{l! (r-l)!} (\gamma^T \nu_j)^{r-l} (\gamma^T x)^l a_j(x) \\
&\leq  r (\gamma^T x)^{r-1} \gamma^T \nu_j a_j(x) + H'_r (\|x\|^r  +1),\\
&=  r (\gamma^T x)^{r-1} \gamma^T \nu_j a_j(x) + H'_r (\|x\|^r \wedge K +1),
\end{aligned}
\]
where $H'_r$ is a suitable constant. We note that for some $j$ if $\gamma^T \nu_j=0$ then there are no conditions on the form of the propensity function $a_j$.   
Otherwise the quadratic growth bound on $a_j$ ensures that an upper bound 
with highest power of at most $\|x\|^r$ is obtained. This leads to a bound of the form
\[
(\sA f_r^K)(x) \leq \lambda_r f_r^K + \lambda_r,
\]
where $\lambda_r>0$ is a suitable constant. The rest of the proof is similar to that of Theorem \ref{thm1_mom_bnd}.
\end{proof}

The next theorem provides a necessary condition for the boundedness of all moments for all time $t \geq 0$. 

\begin{theorem}
\label{thm3}
Suppose propensity functions all belong to $\scP$ and 
suppose that there exist $\gamma \in \real^N$, $\alpha >1$ and $C>0$ such that 
$\gamma>0$ and
\[
\gamma^T F(x) \geq C \|x\|^\alpha, \quad \forall x \in \posint^N,
\]
where $F(x) = \sum_{j=1}^M \nu_j a_j(x)$. Further suppose that $0 \in \posint^N$ is not 
both the initial and an absorbing state.  Then for every $r \in \nat$ 
that satisfies $a_j(x) \leq H (\|x\|^r + 1)$ for all $j=1,\dots,M$ (for some $H$ independent of $x$) there exists $t >0$ such that  $E(\|X(t)\|^r)=\infty$. (As always we assume deterministic initial condition $x_0$.)
\end{theorem}

\begin{proof}
We shall prove by contradiction. 
Suppose $r \in \nat$ satisfies $a_j(x) \leq H (\|x\|^r + 1)$ for all $j=1,\dots,M$ (for some $H$ independent of $x$) and assume that 
for all $t \geq 0$ it holds that $E(\|X(t)\|^r) < \infty$. Let $a_0 = \sum_{j=1}^M a_j$. Then $a_0 \in \scP_r$, $E(a_0(X(t))) < \infty$ for $t \geq 0$, and using Lemma \ref{lem_mom_smooth} it follows that $E(a_0(X(t)))$ is continuous in $t$. Thus we also have that 
\[
\int_0^t E( a_0(X(s))) ds < \infty.
\]
If the number of events of type $j$ occurring during $(0,t]$ is denoted by $R_j(t)$ then 
\[
E(R_j(t)) = \int_0^t E( a_j(X(s))) ds < \infty,
\] 
and hence $R_j(t) < \infty$ with probability $1$. In other words the process is non-explosive. 

First choose a norm such that $\|x\|=\gamma^T |x|$ ($|x|=(|x_1|,\dots,|x_N|)$.
By equivalence of norms, the inequality $\gamma^T F(x) \geq C \|x\|^\alpha$ still holds with possibly a different $C$. 
For each $K>0$  define $M_K>0$ as follows:
\[
M_K = \sup\{ \|x+\nu_j\|, \|x\| \, | \, j=1,\dots,M, \, x \in \posint^n, \, \|x\| \leq K \}.
\]
Clearly as $K \to \infty$, we have that $M_K \to \infty$. 

For each $K>0$ let us introduce the function $f^K: \posint^N \to \real$ by 
$f^K(x) = \|x\| \wedge M_K$ and as $K \to \infty$, we have that $f^K(x) \to \|x\|$. 
We also have that
\[
(\sA f^K)(x) = \sum_{j=1}^M \gamma^T \nu_j a_j(x) = \gamma^T F(x), 
\]
for all $x \in \posint^N$ that satisfy $\|x\| \leq K$.  

We define the stopping times $\tau_K$ by
\[
\tau_K = \inf \{ t \, | \, \|X(t) \| > K \}.
\]
By the non-explosivity, we have that $\tau_K \to \infty$ with probability $1$ as $K \to \infty$. 

We observe that \eqref{eq_Dynkin_int} holds if $t$ is replaced by a bounded stopping time \cite{EK-book}. Since $t \wedge \tau_K$ is bounded above by $t$ we have that 
\[
E(f^K(X(t \wedge \tau_K))) = f^K(x_0) + \sum_{j=1}^M \gamma^T \nu_j E\left( \int_0^{t \wedge \tau_K} a_j(X(s)) ds \right).
\]
For sufficiently large $\nu_0>0$, we can bound 
\[
\|\sum_{j=1}^M \gamma^T \nu_j a_j(x) \| \leq \nu_0 a_0(x),
\]
for all $x \in \posint^N$. Since as $K \to \infty$, we have that $X(t \wedge \tau_K) \to X(t)$ with probability $1$, by dominated convergence theorem, we may 
take limit both sides of above integral equation to obtain that
\[
E(\|X(t)\|) = f(x_0) + \sum_{j=1}^M \gamma^T \nu_j E\left( \int_0^{t} a_j(X(s)) ds \right).
\]
Hence we have that
\[
\frac{d}{dt} E(\|X(t)\|) = E(\gamma^T F(X(t))) \geq C E(\|X(t)\|^\alpha) \geq C (E(\|X(t)\|))^\alpha,
\]
(the last step uses Jensen's inequality, see \cite{WilliamsBook} for instance).  
Let $\phi(t)=E(\|X(t)\|)$. 
Thus $\phi$ satisfies 
\[
\frac{d}{dt} \phi(t) \geq C (\phi(t))^\alpha, \quad t \geq 0. 
\]
Under the assumption on the initial state, for every $t>0$, $\phi(t)>0$. 
Fix $t_1>0$ to obtain $\phi(t) \geq \phi(t_1) >0$ for $t \geq t_1$. 
We obtain for $t \geq t_1$, 
\[
\frac{d}{dt} \left(\frac{-(\alpha-1)}{(\phi(t))^{\alpha-1}}\right) = \frac{\frac{d}{dt} \phi(t)}{(\phi(t))^\alpha} \geq C
\]
and after some manipulations we obtain
\[
(\phi(t))^{\alpha-1} \geq \frac{\alpha-1}{\frac{\alpha-1}{(\phi(t_1))^{\alpha-1}}-Ct},
\]
showing that $\phi(t)=E(\|X(t)\|)$ is not finite for all $t >0$  
reaching a contradiction.
\end{proof}

We like to remark that the condition stated in Theorem \ref{thm3} has 
not been shown to imply the nonexistence of the first order moment $E(\|X(t)\|)$ for all time $t \geq 0$. 

%A careful examination of the proof shows that 
%the condition implies that $E(\|X(t)\|^r)$ may not exist for all $t \geq 0$ where $r>1$ is such that it provides a growth bound $a_0(x) \leq C (\|x\|^r + 1)$.  

{\bf Example 5} Let us consider a modified version of Example 2 where $\nu_1=(2,-1)^T$, $\nu_2=(-1,1)^T$ and $a_1(x)=x_2^2$ as before, but we set $a_2(x) = x_1^2$. The quadratic form for $a_2$ makes reaction $2$ also to be critical. 
In order to satisfy the sufficient condition of Theorem \ref{thm1_mom_bnd} 
we must find $\gamma \in \posint^2$ with $\gamma>0$ and $\gamma \nu \leq 0$. 
This requires the conditions 
\[
\gamma_1>0, \; \; \gamma_2>0, \; \; 2 \gamma_1 \geq \gamma_2, \; \; \gamma_1 \leq \gamma_2,
\]  
which cannot be met. In this example the sufficient conditions of Theorem \ref{thm2} also lead to the same conditions on $\gamma$ which cannot be met. 
On the other hand if we choose $\gamma = (2,3)^T$ then we obtain that 
\[
\gamma^T F(x) = a_1(x) + a_2(x) = x_2^2 + x_1^2 = \|x\|^2.
\]
Thus the condition of Theorem \ref{thm3} is satisfied (assuming initial condition is not $0$) and since propensities are quadratic we can conclude that $E(\|X(t)\|^2)=\infty$ for some $t>0$ .

\bibliography{kinetics}
\bibliographystyle{plain}

\section*{Appendix}
We summarize some basics from convex analysis. A set $C \subset \real^n$ is said to be {\em convex} if for each $x,y \in C$ and each $\alpha \in (0,1)$ it holds
 that $\alpha x + (1-\alpha) y \in C$. A set $K \subset \real^n$ is said to be a {\em cone} with vertex $x \in \real^n$ is for each $y \in K$ and each $\lambda \geq 0$ it holds that $x + \lambda (y-x) \in K$. A {\em convex cone} is simply a set that is both convex and a cone. A cone is said to be {\em finitely generate} or {\em polyhedral} if there exists a finite set of vectors $v_1,\dots,v_m \in \real^n$ 
such that $y \in K$ if and only if there exist $\lambda_1,\lambda_2,\dots,\lambda_n$ all greater than or equal to zero such that 
\[
y = x + \sum_{j=1}^n \lambda_j v_j,
\] 
where $x$ is the vertex. 

It is easy to show that polyhedral cones are closed and convex sets, i.e.\ closed convex cones. It may be shown that if $C,K \subset \real^n$ are polyhedral cones with vertex $0$ then so are $C \cap K$ and $C + K$ where the sum of two sets is defined by
\[
C+K = \{ c + k \, | \, c \in C, \, k \in K\}.
\]
The {\em polar} of a cone $K$ with vertex $0$ is the cone $K^o$ (with vertex $0$) defined by
\[
K^o = \{ y \in \real^n \, | \, \forall x \in K, \; y^T x \leq 0 \}.
\] 

\begin{lemma} \label{lem_polar_cap} Suppose $C$ and $K$ are convex cones with vertex $0$ in $\real^n$. Then 
\[
(C \cap K)^o = \text{cl}(C^o+K^o).
\]
Here $\text{cl}$ refers to the closure of a set.
\end{lemma}
\begin{proof}
See \cite{Florenzano-LeVan-Book} Exercise 2.12.
\end{proof}

\end{document}